\documentclass[11pt]{amsart}
\usepackage{amsmath}
\usepackage{amssymb}
\usepackage{amsfonts}
\usepackage{graphicx}
\usepackage{mathtools}
\usepackage{xcolor}
\usepackage{nameref}
\usepackage{hyperref}
\DeclareSymbolFont{AMSb}{U}{msb}{m}{n}
\DeclareMathSymbol{\I}{\mathbin}{AMSb}{"49}
\DeclareMathSymbol{\C}{\mathbin}{AMSb}{"43}

\newcommand{\FS}{\operatorname{FS}}
\newtheorem{theorem}{Theorem}[section]
\newtheorem{lemma}[theorem]{Lemma}

\theoremstyle{definition}
\newtheorem{definition}[theorem]{Definition}

\theoremstyle{corollary}
\newtheorem{corollary}[theorem]{Corollary}
\theoremstyle{example}

\theoremstyle{note}

\theoremstyle{notation}

\numberwithin{equation}{section}
\newtheorem*{theorem*}{Theorem}
\theoremstyle{question}
\newtheorem{question}[theorem]{Question}

\begin{document}
\title{Monochromatic Sums and Quotients Near Zero}

\author{Md Moid Shaikh}
\address{Md Moid Shaikh, Department of Mathematics, Maharaja Manindra Chandra College, 20 Ramkanto
Bose Street, Kolkata-700003, West Bengal, India}
\email{mdmoidshaikh@gmail.com\\
mdmoidshaikh@mmccollege.ac.in}

\author{Sourav Kanti Patra}
\address{Sourav Kanti Patra, Department of Mathematics, Kishori Sinha Mahila College, Q976+WXP, Aurangabad, Bihar 824101, India}
\email{souravkantipatra@gmail.com}

\author{Mukesh Kumar}
\address{Mukesh Kumar, Department of Mathematics, Magadh University, Bodh Gaya, Aurangabad, Bihar 824101, India}
\email{mukeshmaths95@gmail.com}

%\thanks{}


\begin {abstract}
Recently S. Goswami proved that whenever the set $\mathbb N$ of natural numbers is finitely colored, the set $\{a, b, ab, b(a+1)\}$ is monochromatic which also  established a variant of the long-standing Hindman’s conjecture, which asks for a monochromatic set of the form $\{a, b, ab, a+b\}$. Actually he disproved a conjecture proposed by J. Sahasrabudhe that $\{a, b, a(b + 1)\}$ is not partition regular. In this paper we prove that  $\{a, b, ab, b(a+1)\}$ is monochromatic near zero which means for every finite coloring of a dense subsemigroups of $((0, \infty), +)$,  the set $\{a, b, ab, b(a+1)\}$ is monochromatic near zero or in other words, we will get $a, b$ in a dense subsemigroups of $((0, \infty), +)$ as small as we want such that the set $\{a, b, ab, b(a+1)\}$ is monochromatic for every finite coloring of that dense subsemigroups of $((0, \infty), +)$, also we show that the pattern $x, y, x+y, \frac{y}{x}$ is partition regular near zero.

AMS subjclass [2020]: 05D10.
\end{abstract}

\maketitle

\section{Introduction}
A typical result in Ramsey Theory asserts that a given configuration
will be completely contained in one of the partition classes for
any finite partition of some sufficiently  large or rich structure or in the language
of graph theory, Ramsey Theory asserts that a given configuration
will be monochromatic for any finite coloring of some sufficiently  large or rich structure. One can
say that Ramsey Theory is the study of the preservation of configuration under
set partitions, i.e., Ramsey Theoretic properties are partition regular.

In 1916, I. Schur \cite{sch} proved that there exist $a, b\in \mathbb N$ such that $\{a, b, a+b\}$
is partition regular and it is easy to prove that $\{a, b, a\cdot b\}$ is partition regular by the
mapping $n \longmapsto 2^n$. These two results are
commonly referred to as the additive Schur theorem and the multiplicative Schur theorem, respectively.
It was unknown for more than a century  whether these two phenomena could be unified in
a single statement. In 1979, R. L. Graham and N. Hindman independently proved  such a result
for two-colorings (see \cite[pp. 68–69]{gra}, \cite[Section 4]{hind79}). Consequently, Hindman conjectured that
the statement should hold for any finite colorings which is now known as the Hindman conjecture. In 1917, J. Moreira \cite{mor2017} answered this conjecture partially that there exist $a, b\in \mathbb N$ such that $\{a, a\cdot b, a+b\}$
is partition regular. For the concise proof, one can see \cite{ryan}. In 2017, a conjecture proposed by J. Sahasrabudhe that $\{a, b, a(b + 1)\}$ is not partition regular(see \cite[Question 31.]{jul}). Recently, S. Goswami  \cite{sayan25}
disproved the conjecture of Sahasrabudhe, in fact,  Goswami  establish a stronger result which says that  the larger set $\{a, b, ab, b(a+1)\}$ can always be found monochromatic for any finite coloring of $\mathbb N$. In \cite{naso}, the authors discussed the partition regularity of the configuration $x, y, x+y, \frac{y}{x}$ in $\mathbb N$.

In 1999, Hindman and Leader first brought the concept of the semigroup consisting of ultrafilters on $(0, 1)$ converging to zero and they studied some strong results in Ramsey Theory in near zero, one can see \cite{hindult} for details. Then many results in near zero have been obtained in \cite{mine near 0, bayat, de2012, sou20}. Motivated by these, in this paper, we want to establish Goswami's Theorem in zero, i.e., we want to establish that $\{a, b, ab, b(a+1)\}$
is monochromatic near zero for dense subsemigroups $S$ of $((0, \infty), +)$ which means  that for any finite partition of $S\cap(0,1)$, at least one partition cell contains $a, b, ab, b(a+1)$, also we show that the pattern $x, y, x+y, \frac{y}{x}$ is partition regular near zero. We have designed this paper in the following way: In section 2, we have presented basic results which are used to prove our main results. In Section 3, we have proved Goswami's Theorem near zero using the combinatorial arguments. In the final section, we have proved that  the pattern $x, y, x+y, \frac{y}{x}$ is partition regular near zero based on algebraic arguments which gives Goswami's Theorem near zero.

\section{Basic results}
In this section, we present  some preliminary concepts, definitions, theorems, conventions, and results which will be
used frequently later in this paper. We only give the references to where interested readers can find the detailed descriptions.

At first we give a brief introduction of  the Stone-\v{C}ech compactification of a discrete semigroup. Let $(S,\cdot)$  be any discrete semigroup then the Stone-\v{C}ech compactification $\beta S$ of the discrete
semigroup $(S,\cdot)$ is defined to be the set of all ultrafilters on $S$. The principal ultrafilters are identified with the
points of $S$. For a subset  $A$ of $S$, we denote $\bar{A}=\{ p\in \beta S : A\in p\}$. Then the set $\{ \bar{A} :
A\subseteq S\}$ forms a  clopen basis for a topology on $\beta S$. One can extend the operation $\cdot$ on $S$  to the
Stone-\v{C}ech compactification $\beta S$ of $S$ which makes  $(\beta S,\cdot)$ is a compact,  right topological
semigroup (meaning that for any $p\in \beta S$, the function $\rho_p:\beta S \rightarrow \beta S$ defined by
$\rho_p(q)=q\cdot p$ is continuous) with $S$ contained in its topological center (meaning that for any $x\in S$
the function $\lambda_x : \beta S \rightarrow \beta S$
defined by $\lambda_x (q)=x\cdot q$ is continuous). Let $p,q \in \beta S $, and
$A \subseteq S$, $A \in p\cdot q$ if and only if $\{ x \in S :x^{-1}A \in q\} \in p$, where
$x^{-1}A=\{ y \in S: x\cdot y \in A \}$.

A nonempty subset $I$ of a semigroup $(T,\cdot)$
is said to be a {\it left ideal} of $S$ if $T\cdot I \subseteq I$, a {\it right ideal} of $S$ if
$I\cdot T \subseteq I$, and a {\it two-sided ideal} (or simply an {\it ideal}) if it is both a
left and a right ideal.  A left ideal is {\it minimal} if it does not
contain any proper left ideal. In a similar way, one can define a minimal right ideal
and the smallest ideal. One must have the smallest two-sided ideal for any compact Hausdorff right topological semigroup $(T,\cdot)$.\\
$$\begin{array}{ccc}
K(T) & = & \bigcup\{L:L \text{ is a minimal left ideal of } T\} \\
& = & \,\,\,\,\,\bigcup\{R:R \text{ is a minimal right ideal of } T\}.
\end{array}$$\\
$L\cap R$ is a group for a given a minimal left ideal $L$ and a minimal right ideal $R$ and hence it   contains
an idempotent(an element $a\in S$
is said to be an {\it idempotent} if $a = a\cdot a$). An idempotent belonging to the smallest ideal is minimal and conversely. For more details readers may consult Section 2.2 of \cite{hindalg}. From now on we use $\mathcal{P}_{f}(X)$
to denote the set of all finite nonempty subsets of a set $X$ and $E(S)$ to denote the set of all idempotents of $S$.

Our desired results are derived for an arbitrary subsemigroup of $\mathbb R$ which is dense in $(0, \infty)$, for example, the set of positive rational numbers.
In \cite{hindult}, Hindman and Leader first introduced the semigroup consisting of ultrafilters
converging to zero.   Let $S$ be a dense subsemigroup of $((0, \infty), +)$, then one can define
$ 0^+(S)= \{p\in \beta S:$ for all $(\epsilon>0)((0, \epsilon)\cap S\in p)\}$. It is necessary to
take $S$ as a dense subsemigroup of $((0, \infty), +)$ otherwise $ 0^+(S)$ will be empty.
\begin{theorem} {\normalfont \cite[Lemma 13.29.(f)]{hindalg}} Let $S$ be a dense subsemigroup of $((0, \infty), +)$ such that $S\cap(0,1)$ is a
subsemigroup of $((0, 1), \cdot)$ and assume that for each $y\in S\cap(0,1)$ and each $x\in S$, $x/y$
and $yx\in S$. Then $ 0^+(S)$ is a two-sided ideal of $\beta(S\cap(0,1), \cdot)$.
\label{2.1}
\end{theorem}
As $ 0^+(S)$ is a two-sided ideal of $(\beta(0,1)_d, \cdot)$, so $K(\beta(0,1)_d, \cdot)\subseteq 0^+(S)$
and therefore by \cite[Theorem 1.65]{hindalg}, $K(0^+(S), \cdot)=K(\beta(0,1)_d, \cdot)$.  As $0^+(S)$ is a compact Hausdorff right topological semigroup, by
\cite[Theorem 1.3.11]{berg89}, $0^+(S)$ contains a smallest two-sided ideal which is denoted by
$K(0^+(S))$ and it must contain an
idempotent by \cite[Corollary 2.10]{elli}.
It is to be observed that $0^+(S)\cap K(\beta S)=\emptyset$ and hence $0^+(S)$ provides some new
information that is not available from $K(\beta S)$. After
introducing the notion of a semigroup of ultrafilters near
zero or converging to $0$, Hindman and Leader studied some Ramsey Theoretic results near zero in
\cite{hindult}. In \cite{sou20}, Patra and Shaikh gave the name HL semigroups for the semigroups satisfying the hypothesis of the Theorem \ref{2.1}.

\begin{definition}{\normalfont \cite[Definition 5]{sou20} Let $S$ be a dense subsemigroup of $((0, \infty), +)$. Then $S$ is an {\it HL semigroup}
if and only if $S\cap(0,1)$ is a subsemigroup of $((0, 1), \cdot)$ and   for each $y\in S\cap(0,1)$ and
for each $x\in S$, $x/y$ and $yx\in S$.}
 \label{2.2}
 \end{definition}
 In an HL semigroup the requirement $y < 1$ is not essential, but makes the proof of theorems, lemmas,
corollaries related to HL semigroup simpler since under that assumption, if $x < 1/n$, then $yx < 1/n$.

 In \cite{furst}, Furstenberg introduced a number of classes of large sets originating from
topological dynamics. Alternative characterizations of such sets are also available in terms
of the algebraic structure of $\beta\mathbb{N}$.
Now we introduce the IP sets and central sets which have rich combinatorial properties. There is a nice intimate relation between IP sets and idempotents. We need the following definition to move forward.

 \begin{definition} {\normalfont Let $\langle{x_n}\rangle_{n=1}^{\infty}$ be an infinite sequence of
positive real numbers. Then
$FS(\langle{x_n}\rangle_{n=1}^{\infty})=\{\sum_{n\in F}x_n:F\in\mathcal{P}_f(\mathbb N)\}$.}

 \label{2.3}
 \end{definition}

\begin{definition} {\normalfont  A subset  $A$ of $\mathbb N$ is said
to be an $IP${\it-set} if and only if  there exists a sequence $\langle{x_n}\rangle_{n=1}^{\infty}$ in $\mathbb N$
such that  $FS(\langle{x_n}\rangle_{n=1}^{\infty})\subseteq A$.}
\label{2.4}
\end{definition}

\begin{definition} {\normalfont Let $(S,\cdot)$ be a semigroup and let $A\subseteq S$.
$A$ is called {\it Central} in $(S,\cdot)$ if there is some idempotent $p\in K(\beta S)$
such that $A\in p$.}\label{new 2.4}
\end{definition}

The following definition is the near zero version of the IP-set in $\mathbb N$. The IP-set near zero is extremly important object which will be used to prove our desired results in the next section.
\begin{definition} {\normalfont Let $S$ be a dense subsemigroup of $((0, \infty), +)$. A subset  $A$ of $ S$ is said
to be an $IP${\it-set near zero} if and only if  there exists a sequence $\langle{x_n}\rangle_{n=1}^{\infty}$
such that $\sum_{n=1}^{\infty}x_{n}$ converges and $FS(\langle{x_n}\rangle_{n=1}^{\infty})\subseteq A$.}
\label{2.5}
\end{definition}

The following Theorem shows that IP sets near zero are members of idempotents living near zero.

\begin{theorem}{\normalfont \cite[Theorem 3.1]{hindult}} Let $S$ be a dense subsemigroup of $((0, \infty), +)$ and let  $A\subseteq S$. Then $A$
is an $IP$-set near zero  if and only if there is some idempotent $p$ in $0^+(S)$ such that $A\in p$.
\label{2.6}
\end{theorem}

The following definition is the near zero version of central set.
\begin{definition} {\normalfont Let $S$ be an HL semigroup. A subset  $A$ of $ S$ is said
to be an $Cental${\it-set near zero} if there is some idempotent $p\in K(0^+(S)$
such that $A\in p$.}
\label{2.7}
\end{definition}

In \cite{Bergelson1} Bergelson and Glasscock investigated the interplay between additive and
multiplicative largeness. The following theorem shows combined additive and multiplicative structure can be found in one partition cell for any finite partition of $(S\cap(0,1), \cdot)$.

\begin{theorem}{\normalfont \cite[Theorem 5.6]{hindult}} Let $S$ be an HL semigroup. Let $r\in \mathbb{N}$ and let $S\cap(0,1) = \bigcup_{i=1}^{r}B_{i}$.
Then there is some
$i\in \{1,2,\ldots,r\}$ such that $B_{i}$ is central near zero and $B_{i}$ is central in $(S\cap(0,1),\cdot)$.
\label{2.8}
\end{theorem}

As a consequence of the above theorem we have the following corollary.

\begin{corollary}{\normalfont \cite[Corollary 1]{sou20}} Let $S$ be an HL semigroup. Let $r\in \mathbb{N}$ and let
$S\cap(0,1) = \bigcup_{i=1}^{r}B_{i}$.
Then there is some
$i\in \{1,2,\ldots,r\}$ such that $B_{i}$ is an $IP$-set near zero and $B_{i}$ is central in
$(S\cap(0,1),\cdot)$.
\label{2.9}
\end{corollary}

\section{Product and Translated Product Near zero}

In \cite{sayan25}, Goswami proved that $\{a, b, ab, b(a+1)\}$ is monochromatic for every finite coloring of $\mathbb N$. He used the technique of the
proof similar to the argument of Moreira \cite{mor2017}, who employed van der Waerden’s theorem \cite{waerden}, whereas
 Goswami made iterative use of Hindman’s theorem. In this section, our main target is to show that $\{a, b, ab, b(a+1)\}$
is monochromatic near zero for dense subsemigroups of $((0, \infty), +)$ which means  we have to show that for any finite partition of $S\cap(0,1)$, at least one partition cell contains $a, b, ab, b(a+1)$. We shall follow the similar arguments made by Goswami in \cite{sayan25}. Before going to our main theorem, we need the following definition.

\begin{definition} {\normalfont Let $S$ be an HL semigroup and let  $A\subseteq S$. Then we denote
\begin{enumerate}
 \item $-y+A=\{x\in S: x+y\in A\}$,
 \item $y^{-1}A=\{x\in S: xy\in A\}$.
 \end{enumerate}}
\label{3.1}
\end{definition}
 We need three following lemmas to prove our main theorem. The following lemma is a near zero version of reformulated
Hindman's theorem \cite{hind74}.
\begin{lemma}
 Let $S$ be an HL semigroup. Let  $A\subseteq S$ be an IP-set near zero, then there exists $y\in A$ such that $(-y+A)\cap A$ is also an IP-set near zero.
\label{3.2}\end{lemma}
\begin{proof}
 Since $A\subseteq S$ is an IP-set near zero, there exists $p\in E(0^+(S))$ such that $A\in p$. Let $A^*=\{x\in A: -x+A\in p\}$, then $A^*\in p$. Take $y\in A^*$, then $-y+A^*\in p$. So, $(-y+A^*)\cap A\in p$ and thus $(-y+A)\cap A\in p$. Therefore, $(-y+A)\cap A$ is  an IP-set near zero.
\end{proof}

\begin{lemma}
 Let $S$ be an HL semigroup. Let  $A\subseteq S$ be an IP-set near zero, then for any $y\in S\cap (0, 1)$,  $y^{-1}A$ is also an IP-set near zero.
\label{3.3}\end{lemma}
\begin{proof}
 Let $p$ be an idempotent of $0^+(S)$ such that $A\in p$. The function $f: S\longrightarrow S$ given by $f(x)=y^{-1}x$ is an injective homomorphism and so is its continuous extension $\tilde{f}: \beta S\longrightarrow \beta S$ by \cite[Corollary 4.22]{hindalg} and \cite[Exercise 3.4.1]{hindalg}. So, $y^{-1}p$ is an idempotent of $0^+(S)$ and $y^{-1}A\in y^{-1}p$. Thus $y^{-1}A$ is  an IP-set near zero.
\end{proof}

\begin{lemma}
 Let $S$ be an HL semigroup. Let $r\in \mathbb N$ and $S=\bigcup_{i=1}^rA_i$ be a finite r-coloring of $S$. Let  $A\subseteq S$ be an IP-set near zero, then there exists $i\in \{1, 2, 3, \ldots, r\}$ such that $A\cap A_i$ is also an IP-set near zero.
\label{3.4}\end{lemma}
\begin{proof}
 Choose $p\in E(0^+(S))$ such that $A\in p$. Now $\bigcup_{i=1}^rA_i=S\in p$ and so $A_i\in p$ for some  $i\in \{1, 2, 3, \ldots, r\}$. Thus $A\cap A_i\in p$ for some  $i\in \{1, 2, 3, \ldots, r\}$. Therefore, $A\cap A_i$ is  an IP-set near zero.
\end{proof}
Now we state Goswami's theorem and next we give its near zero version.

\begin{theorem}{\normalfont \cite[Theorem 1.1]{sayan25}} For any finite coloring of $\mathbb N$, there exist $a, b\in \mathbb N$ such that the set $\{a, b, ab, b(a+1)\}$ is monochromatic.

\label{3.5}\end{theorem}

\begin{theorem} Let $S$ be an HL semigroup. For any finite coloring of $S$, there exist $a, b\in S$ such that the set $\{a, b, ab, b(a+1)\}$ is monochromatic near zero.

\label{3.6}\end{theorem}
\begin{proof}
 Let $\epsilon>0$ be given. Without loss of generality, assume that $0<\epsilon<1$. Suppose that $S$ is finitely colored. $S=\bigcup_{i=1}^rA_i$. Without loss of generality, assume that $A_1$ s  an IP-set near zero, and set $D_0=A_1$. By Lemma \ref{3.2}, we can choose an element $y_1\in A_1\cap (0, \epsilon/2)$ such that $(-y_1+A_1)\cap A_1$ is  an IP-set near zero. Then by Lemma \ref{3.3}, $C_1=y_1^{-1}((-y_1+A_1)\cap A_1)$ is  an IP-set near zero. Finally by Lemma \ref{3.4},  there exists $i\in \{1, 2, 3, \ldots, r\}$ such that $D_1=C_1\cap A_{i_1}=y_1^{-1}((-y_1+A_1)\cap A_1)\cap  A_{i_1}$ is  an IP-set near zero.

 Using Lemma \ref{3.2} and Lemma \ref{3.3}, choose $y_2\in D_1\cap (0, \epsilon/2)$ such that $C_2=y_2^{-1}((-y_2+D_1)\cap D_1)$ is  an IP-set near zero. Then by Lemma \ref{3.4}, there exists $i\in \{1, 2, 3, \ldots, r\}$ such that $D_2=C_2\cap A_{i_2}$ is  an IP-set near zero. Proceeding inductively, we define for each $n\in \mathbb N$:
 \begin{enumerate}
  \item Choose $y_n\in D_{n-1}\cap (0, \epsilon/2)$ and $i_n\in \{1, 2, 3, \ldots, r\}$ such that
  \item the set $D_n=y_n^{-1}((-y_n+D_{n-1})\cap D_{n-1})\cap  A_{i_n}$ is  an IP-set near zero.
  \end{enumerate}
  So, there exists $k\in \{1, 2, 3, \ldots, r\}$ such that $A_{i_j}=A_{i_n}=A_k$ for some indices $j<n$. Now choose $x\in D_n\cap (0, \epsilon/2)\subseteq A_{i_n}=A_k$.

  From the construction in (2), we have $xy_n\in D_{n-1}\subseteq y_{n-1}^{-1}D_{n-2}\subseteq\cdots \subseteq y_{n-1}^{-1}\cdots y_{j+1}^{-1}D_j$. Thus $xy_ny_{n-1}\cdots y_{j+1}\in D_j\subseteq A_{i_j}=A_k$.

  Moreover, since $y_n\in D_{n-1}$, applying the same argument, we have $y_ny_{n-1}\cdots y_{j+1}\in D_j\subseteq A_{i_j}=A_k$.

  Thus, both $y_ny_{n-1}\cdots y_{j+1}$ and $xy_ny_{n-1}\cdots y_{j+1}$ lie in the same color class $A_k$.

  Again, observe that $xy_n+y_n\in D_{n-1}$, so by the same argument as above $(xy_n+y_n)y_{n-1}\cdots y_{j+1}\in D_j\subseteq A_{i_j}$.

  Now define

  (1) $a=x$

  (2) $b=y_ny_{n-1}\cdots y_{j+1}$

  Then $a, b\in A_k$, and $xy_ny_{n-1}\cdots y_{j+1}=ab\in A_k$ and $(xy_n+y_n)y_{n-1}\cdots y_{j+1}=(a+1)b\in A_k$.

  Also $0<x<\epsilon/2$ and $0<y_t<\epsilon/2$, for $j+1\leq t\leq n$ imply $0<a<\epsilon/2<\epsilon$,
$0<b<\epsilon/2<\epsilon$ and hence $0<ab<\epsilon/2<\epsilon$ and $0<(a+1)b<\epsilon$.

Therefore, the set $\{a, b, ab, b(a+1)\}$ is monochromatic near zero.

This completes the proof.
\end{proof}

\section{Sums and Quotients Near zero}

In the last section, we have proved our desired results using the combinatorial arguments. Here in this present section, we
give more strengthened results of the results discussed in the last section using  tensor product of ultrafilters, i.e., our results of this present section  will be based on algebraic approach. Let us begin with by recalling the definition of tensor product.
\begin{definition} {\normalfont \cite[Definition 1.1]{hindtensor}}
 Let $X$ and  $Y$ be discrete spaces, let $p\in \beta X$, and let $q\in \beta Y$. Then the tensor product of $p$ and $q$ is defined by $p\otimes q=\{A\subseteq X\times Y: \{x\in X: \{y\in Y: (x, y)\in A\}\in q\}\in p\}$.
\end{definition}
It is easy to see that $ A\times B\in p\otimes q$ if and only if $A\in p$ and $B\in q$. Now we define a mapaping which will connect the tensor product with $0^+(S)$ for an HL semigroup.
\begin{definition}{\normalfont Let $S$ be an HL semigroup and $s_0\in S$ be a fixed element of $S$. Define a map $D: S\times S\longrightarrow S$ by $D(x, y)=\begin{cases}  \frac{y}{x}, \text{ if } x\in S\cap (0, 1), \\ s_0, \text{ otherwise}. \end{cases}$}
\end{definition}
The following lemma shows that the image of a tensor product of $u\in 0^+(S)$ with itself under $D$ is in $0^+(S)$.
\begin{lemma}\label{4.2}
 Let $S$ be an HL semigroup. If  $u\in 0^+(S)$, then $D(u\otimes u)\in 0^+(S)$.
\end{lemma}
\begin{proof}
 Let $\epsilon>0$ be given.
 Now $D^{-1}[S\cap (0, \epsilon)]=\{(x, y)\in S\times S: D(x, y)\in S\cap (0, \epsilon)\}\supseteq\{(x, y)\in S\times S: \frac{y}{x}\in S\cap (0, \epsilon)\}=\{(x, y)\in S\times S: y\in x(S\cap (0, \epsilon))\}$.

 Observe that $\{(x, y)\in S\times S: y\in x(S\cap (0, \epsilon))\}\in u\otimes u$ if and only if $\{x\in S: \{y\in S: y\in x(S\cap (0, \epsilon))\}\in u\}\in u$ if and only if $\{x\in S: \{ x(S\cap (0, \epsilon))\}\in u\}\in u$. Since for each $x\in S\cap (0, 1)$, $x(S\cap (0, \epsilon))\in u$, so we have $D^{-1}[S\cap (0, \epsilon)]\in u\otimes u$. Thus $S\cap (0, \epsilon)\in D(u\otimes u)$ for all $\epsilon>0$ and so $D(u\otimes u)\in 0^+(S)$.
\end{proof}
 We denote $F<G$ for any $F, G\in P_f(\mathbb N)$ if $\max F < \min G$.
\begin{lemma}\label{4.1}
 Let $S$ be an HL semigroup and $A\subseteq S\times S$. Then the following statements are equivalent:
 \begin{enumerate}
  \item There is an idempotent $u\in 0^+(S)$ such that $A\in u\otimes u$.
  \item There is a convergent sequence $\langle{x_n}\rangle_{n=1}^{\infty}$ in $S$ such that $\langle{\frac {x_{n+1}}{x_n}}\rangle_{n=1}^{\infty}$ is  convergent and $\{(\sum_{n\in F}x_n, \sum_{n\in G}x_n): F, G\in P_f(\mathbb N) \text{ and } F<G\}\subseteq A$
 \end{enumerate}

\end{lemma}
\begin{proof}
 This lemma is \cite[Theorem 4.13]{sou23} for $m=2$.
\end{proof}

We are now in a position to prove our main result of this section. The proof of the following theorem is adopted from the proof of \cite[Theorem 1.4]{naso}.

\begin{theorem}\label{4.3}
 Let $S$ be an HL semigroup and $u\in 0^+(S)$ be such that $u=u+u$. Then for any $C\in D(u\otimes u)$, there exist sequences $\langle{x_n}\rangle_{n=1}^{\infty}$ and $\langle{y_n}\rangle_{n=1}^{\infty}$ such that
 \begin{enumerate}
  \item $\FS (\langle{x_n}\rangle_{n=1}^{\infty})\cup \{\frac{\sum_{n\in G}x_n}{\sum_{n\in F}x_n}: F, G\in P_f(\mathbb N) \text{ and } F<G\}\subseteq C$ and
  \item $\{\sum_{l\in F}\prod_{n=k}^ly_n: F\in P_f(\mathbb N), k\leq\min F\}\subseteq C$.
\end{enumerate}
\end{theorem}
\begin{proof}
 Since $C\in D(u\otimes u)$, we have $D^{-1}[C]\in u\otimes u$. Now by Lemma \ref{4.1}, we get a convergent sequence $\langle{z_n}\rangle_{n=1}^{\infty}$ with $\langle{\frac {z_{n+1}}{z_n}}\rangle_{n=1}^{\infty}$ is  convergent such that whenever $F, G\in P_f(\mathbb N) \text{ and } F<G$, we have $(\sum_{n\in F}z_n, \sum_{n\in G}z_n)\in D^{-1}[C]$ so $D(\sum_{n\in F}z_n, \sum_{n\in G}z_n)=\frac{\sum_{n\in G}z_n}{\sum_{n\in F}z_n}\in C$.

 Now set $x_n=\frac{z_{n+1}}{z_1}$ and $y_n=\frac{z_{n+1}}{z_n}$.

 Then both the sequences $\langle{x_n}\rangle_{n=1}^{\infty}$ and $\langle{y_n}\rangle_{n=1}^{\infty}$ are convergent. Also for every $F, G\in P_f(\mathbb N)$, we have $\sum_{n\in F}x_n=\frac{\sum_{n\in F}z_{n+1}}{z_1}$ and
 \\$\frac{\sum_{n\in G}x_{n}}{\sum_{n\in F}x_{n}}=\frac{(\frac{1}{z_1})\sum_{n\in G}z_{n+1}}{(\frac{1}{z_1})\sum_{n\in F}z_{n+1}}=\frac{\sum_{n\in G}z_{n+1}}{\sum_{n\in F}z_{n+1}}$.

 It follows that $\FS (\langle{x_n}\rangle_{n=1}^{\infty})\cup \{\frac{\sum_{n\in G}x_n}{\sum_{n\in F}x_n}: F, G\in P_f(\mathbb N) \text{ and } F<G\}\subseteq C$.

 Moreover, if $F\in P_f(\mathbb N)$ and $k\leq\min F$, we analogously have \\
 $\sum_{l\in F}\prod_{n=k}^ly_n=\sum_{l\in F}\prod_{n=k}^l\frac{z_{n+1}}{z_n}=\sum_{l\in F}\frac{z_{l+1}}{z_k}=\frac{\sum_{l\in F}z_{l+1}}{z_k}\in C$.
\end{proof}
\begin{corollary}
Let $S$ be an HL semigroup. Then for every finite coloring of $S=\bigcup_{i=1}^{r}C_{i}$$, r\in \mathbb{N}$, and $\epsilon>0$, there are $m\leq r$ and convergent sequences $\langle{x_n}\rangle_{n=1}^{\infty}$ and $\langle{y_n}\rangle_{n=1}^{\infty}$ such that \begin{enumerate}
  \item $\FS (\langle{x_n}\rangle_{n=1}^{\infty})\cup \{\frac{\sum_{n\in G}x_n}{\sum_{n\in F}x_n}: F, G\in P_f(\mathbb N) \text{ and } F<G\}\subseteq C\cap (0, \epsilon)$ and
  \item $\{\sum_{l\in F}\prod_{n=k}^ly_n: F\in P_f(\mathbb N), k\leq\min F\}\subseteq C\cap (0, \epsilon)$.
\end{enumerate}
\end{corollary}
\begin{proof}
 Take $u\in 0^+(S)$. Then by Lemma \ref{4.2}, $D(u\otimes u)\in 0^+(S)$. Therefore, choose $m<r$ such that $C_m\cap (0, \epsilon)\in D(u\otimes u)$. Now use Theorem \ref{4.3} to conclude the required results.
\end{proof}
 The following corollary shows that the pattern $x, y, x+y, \frac{y}{x}$ is partition regular near zero.
\begin{corollary}\label{4.4}
Let $S$ be an HL semigroup. Then for every finite coloring of $S=\bigcup_{i=1}^{r}C_{i}$$, r\in \mathbb{N}$, and $\epsilon>0$, there exists  $m< r$, $x, y\in S$ such that $\{x, y, x+y, \frac{y}{x}\}\subseteq C_m\cap (0, \epsilon)$.
\end{corollary}

Note that Goswami's Theorem near zero follows from Corollary \ref{4.4}, by taking $a=x,  b=\frac{y}{x}$.

 In 2011, Hindman proved that for every finite coloring of $\mathbb N$, there must exist $a, b, c$, and $d$ in $\mathbb N$ such that $a, b, c$, and $d$ are all
distinct and that the color of $a+b=cd$ is that same as that of $a, b, c$, and $d$ \cite{hind11}. In 2020, Patra and Shaikh proved this result in near zero \cite{sou20}. In 2017, J. Moreira  prove that for every finite coloring of $\mathbb N$, there exist $a, b\in \mathbb N$ such that the set $\{a, a\cdot b, a+b\}$  is monochromatic \cite{mor2017}. Now we end this work with the following question.

\begin{question}
 Let $S$ be an HL semigroup. For any finite coloring of $S$, there exist $a, b\in S$ such that the set $\{a, a\cdot b, a+b\}$ is monochromatic near zero.
\end{question}

\end{document}